\documentclass[reqno]{amsart}
\usepackage[margin=1in]{geometry}
\usepackage[utf8]{inputenc}
\usepackage[T1]{fontenc}
\usepackage{xcolor}
\usepackage{mathrsfs}
\usepackage{amsthm}
\usepackage{amsmath}
\usepackage{amsfonts}
\usepackage{amssymb}
\usepackage{graphicx}
\usepackage{mathtools}
\usepackage{enumitem}
\usepackage[hidelinks]{hyperref}
\usepackage{tikz-cd}
\usepackage{tikz}

\numberwithin{equation}{section}

\theoremstyle{plain}
\newtheorem{Thm}{Theorem}[section]
\newtheorem{Prop}[Thm]{Proposition}
\newtheorem{Lem}[Thm]{Lemma}
\newtheorem{Cor}[Thm]{Corollary}
\theoremstyle{definition}
\newtheorem{Def}[Thm]{Definition}
\newtheorem{Rem}[Thm]{Remark}

\newcommand{\Nu}{\mathcal{V}}

\newcommand{\defn}[1]{\textbf{\textit{#1}}}
\newcommand{\D}{\mathcal{D}}
\newcommand{\R}{\mathbb{R}}
\newcommand{\G}{\mathcal{G}}
\renewcommand{\H}{\mathcal{H}}
\newcommand{\g}{\mathfrak{g}}

\newcommand{\too}{\longrightarrow}
\newcommand{\mtoo}{\longmapsto}
\newcommand{\Ad}{\operatorname{Ad}}
\newcommand{\pr}{\mathrm{pr}}
\newcommand{\C}{\mathcal{C}}
\renewcommand{\O}{\mathcal{O}}
\newcommand{\M}{\mathcal{M}}
\newcommand{\N}{\mathcal{N}}

\title{Deformations of quasi-Hamiltonian spaces}
\date{June 16, 2025}
\author[Jean-Philippe Burelle]{Jean-Philippe Burelle}
\author[Mohamed Moussadek Maiza]{Mohamed Moussadek Maiza}
\author[Maxence Mayrand]{Maxence Mayrand}
\address[Jean-Philippe Burelle]{D\'{e}partement de math\'{e}matiques \\ Universit\'{e} de Sherbrooke \\ 2500 Bd de l’Universit\'{e} \\ Sherbrooke, QC, J1K 2R1, Canada}
\email{jean-philippe.burelle@usherbrooke.ca}
\address[Mohamed Moussadek Maiza]{D\'{e}partement de math\'{e}matiques \\ Universit\'{e} de Sherbrooke \\ 2500 Bd de l’Universit\'{e} \\ Sherbrooke, QC, J1K 2R1, Canada}
\email{mohamed.moussadek.maiza@usherbrooke.ca}
\address[Maxence Mayrand]{D\'{e}partement de math\'{e}matiques \\ Universit\'{e} de Sherbrooke \\ 2500 Bd de l’Universit\'{e} \\ Sherbrooke, QC, J1K 2R1, Canada}
\email{maxence.mayrand@usherbrooke.ca}

\begin{document}

\begin{abstract}
We introduce a notion of deformations of quasi-Hamiltonian $G$-spaces to Hamiltonian $G$-spaces and provide several examples.
In particular, we show that the double $G \times G$ of a Lie group, viewed as a quasi-Hamiltonian $G \times G$-space, deforms smoothly to the cotangent bundle $T^*G$.
Likewise, any conjugacy class of $G$ sufficiently close to the identity deforms to a coadjoint orbit.
We further show that the moduli space of flat $G$-connections on a compact oriented surface of genus $g$ with $r+1$ boundary components deforms to $T^*G^{r+g}$.
\end{abstract}

\maketitle

\section{Introduction}

Quasi-Hamiltonian geometry \cite{alekseev1998lie} can be thought of as a ``multiplicative analogue'' of Hamiltonian geometry \cite{marsden1974reduction}, where the dual $\g^*$ of the Lie algebra $\g$ is replaced by the group $G$ integrating $\g$.
As such, many natural spaces in Hamiltonian geometry have multiplicative analogues in quasi-Hamiltonian geometry.
A prominent example is the double $D(G) = G \times G$ of a Lie group $G$, which is a quasi-Hamiltonian $G \times G$-space analogous to the cotangent bundle $T^*G$.
Similarly, conjugacy classes in $G$ are quasi-Hamiltonian spaces analogous to the coadjoint orbits in $\g^*$ (see e.g.\ \cite{BM} and the references therein for more examples).
The goal of this paper is to connect such pairs via smooth deformations.

In more detail, we define a Hamiltonian deformation of a quasi-Hamiltonian $G$-space $M$ to a Hamiltonian $G$-space $N$ as a smooth $G$-manifold $\D$ with a submersion $\pi : \D \to \R$, a smooth family of $2$-forms $\widehat{\omega}_t$ on the fibers $\D_t = \pi^{-1}(t)$ and a ``moment map'' $\widehat{\mu} : \D \to (G \times \R^\times) \sqcup (\g \times \{0\})$ (where the codomain is the deformation space of $G$ at $1$ as defined in \S\ref{lq5uqbl3}), such that $\D_t$ smoothly interpolates from $M$ at $t = 1$ to $N$ at $t = 0$; see Definition \ref{swcoikg4} for the precise formulation.

We provide the following examples of Hamiltonian deformations.

\begin{Thm}
Let $G$ be a Lie group equipped with an invariant, non-degenerate, symmetric bilinear form on its Lie algebra $\g$.
\begin{enumerate}[label={\textup{(\roman*)}}]
\item
The double $D(G) = G \times G$ admits a Hamiltonian deformation to the cotangent bundle $T^*G$ \textup{(Theorem \ref{ikr142qd})}.
\item
Let $x \in \g$ be such that $d\exp_{tx}$ is invertible for all $t \in [0, 1]$.
Then the conjugacy class $C_{e^x} \subset G$ admits a Hamiltonian deformation to the coadjoint orbit $\O_x \subset \g \cong \g^*$ \textup{(Theorem \ref{thnlzgh9})}.
\item
Let $\Sigma$ be a compact connected oriented surface of genus $g$ with $r + 1$ boundary components and suppose that $G$ is compact.
The moduli space of flat $G$-connections on $\Sigma$ (resp.\ with holonomies on the boundaries in fixed conjugacy classes) admits a Hamiltonian deformation to $T^*G^{r+g}$ (resp.\ to a symplectic reduction of $T^*G^{r+g}$) \textup{(Corollary \ref{lstnw1kj} and Corollary \ref{rupcqbmj})}.
\end{enumerate}
\end{Thm}

The third part follows from more general results about deformations of fusion products and reductions.

\begin{Thm}
\
\begin{enumerate}[label={\textup{(\roman*)}}]
\item
\textbf{Fusion.}
Suppose that there are Hamiltonian deformations of quasi-Hamiltonian $G \times H_i$-spaces $M_i$ to Hamiltonian $G \times H_i$-spaces $N_i$ for $i = 1, 2$. 
Then there is a Hamiltonian deformation of the fusion product $M_1 \circledast M_2$ to $N_1 \times N_2$ \textup{(Corollary \ref{6mj8ptsk})}.
\item
\textbf{Reduction.}
Suppose that there is a Hamiltonian deformation of a quasi-Hamiltonian $G \times H$-space $M$ to a Hamiltonian $G \times H$-space $N$.
Under suitable conditions, the quasi-Hamiltonian reduction $M_{e^x}$ for $x \in \g$ admits a Hamiltonian deformation to the Hamiltonian reduction $N_x$ \textup{(Theorem \ref{oi2khfoe})}.
\end{enumerate}
\end{Thm}

\subsection{Acknowledgments}
We acknowledge the support of the Natural Sciences and Engineering Research Council of Canada (NSERC) [funding reference numbers RGPIN-2020-05557 and RGPIN-2023-04587]. The second author thanks the Institut des sciences mathématiques ISM and the Ernest-Monga scholarship for their financial support.

\section{Deformations of quasi-Hamiltonian spaces}

The goal of this section is to define Hamiltonian deformations.
We begin with some preliminaries on deformation spaces and quasi-Hamiltonian geometry.

\subsection{Deformation spaces}\label{lq5uqbl3}
Consider a pair $(M, N)$, where $M$ is a smooth manifold and $i : N \hookrightarrow M$ is a smooth submanifold of $M$.
Let $\Nu(M, N)$ be the normal bundle of $N$ in $M$.
Recall that the \defn{deformation space} \cite{MR925720} (see also \cite{MR3846057,bischoff2020deformation}) of $(M, N)$ is the disjoint union
\[
\D(M, N) \coloneqq (M \times \R^\times) \sqcup (\Nu(M, N) \times \{0\})
\]
endowed with the unique topology and smooth structure satisfying the following properties.
\begin{enumerate}[label={(\roman*)}]
\item
The projection
\[
\pi: \D(M, N) \longrightarrow \R,
\quad
(x, t) \mtoo t
\]
is a smooth submersion.

\item
The map
\[
\kappa : \mathcal{D}(M,N) \longrightarrow M,
\quad
(x, t)
\mtoo
\begin{cases}
x & \text{if } t \ne 0 \\
p(x) & \text{if } t = 0
\end{cases}
\]
is smooth, where $p : \Nu(M, N) \to N$ is the projection.

\item 
For every $f \in C^\infty(M)$ such that $f|_N = 0$, the function 
\[
\tilde{f} : \D(M, N) \too \R,
\quad
(x, t)
\mtoo
\begin{cases}
\frac{1}{t}f(x)  & \text{if } t \ne 0 \\
df(x) & \text{if } t = 0
\end{cases}
\]
is smooth.
\end{enumerate}

Of particular importance for this paper is the deformation space
\begin{equation}\label{reboyrqr}
\G \coloneqq \D(G, \{1\}) = (G \times \R^\times) \sqcup (\g \times \{0\})
\end{equation}
of a Lie group $G$ to its Lie algebra $\g$.
The corresponding submersion will be denoted
\[
\pi_{\G} : \G \too \R.
\]
The following result gives convenient coordinates near the zero fiber of $\G$ and will be central to the paper.

\begin{Lem}\label{15sdkw85}
The map
\[
\phi : \g \times \R \too \G, \quad
(x, t) \mtoo
\begin{cases}
(\exp(tx), t) & \text{if } t \ne 0 \\
(x, 0) & \text{if } t = 0
\end{cases}
\]
is smooth.
Moreover, $\phi$ restricts to a diffeomorphism on an open set $U \subset \g \times \R$ containing $\g \times \{0\}$ and $\{0\} \times \R$.
\end{Lem}

\begin{proof}
We can construct coordinates on a deformation space $\D(M, N)$ with functions of the form $\pi$, $\kappa^*f$ for $f \in C^\infty(M)$, and $\tilde{f}$ for $f \in C^\infty(M)$ such that $f|_N = 0$; see \cite[Appendix A]{bischoff2020deformation}.
Hence, to show that $\phi$ is smooth, it suffices to show that $\pi_\G \circ \phi$, $\kappa \circ \phi$, and $\tilde{f} \circ \phi$ are smooth for every $f \in C^\infty(G)$ such that $f(1) =
0$.
We have $(\pi_\G \circ \phi)(x, t) = t$, which is clearly smooth.
Now, for all $(x, t) \in \g \times \R$, we have $(\kappa \circ \phi)(x, t) = \exp(tx)$, so $\kappa \circ \phi$ is also smooth.
Finally, if $f \in C^\infty(G)$ and $f(1) = 0$, then for all $(x, t) \in \g \times \R$ we have
\[
(\tilde{f} \circ \phi)(x, t) = 
\begin{cases}
\frac{1}{t} f(\exp(tx)) & \text{if } t \ne 0 \\
df(x) & \text{if } t = 0,
\end{cases}
\]
so $\tilde{f} \circ \phi$ is also smooth.
We also notice that $d\phi_{(x, t)}(y, u) = (ux + d\exp_{tx}(ty), u)$ for all $(x, t) \in \g \times \R^\times$ and $(y, u) \in \g \times \R$.
Hence, $d\phi_{(x, t)}$ is invertible if $d\exp_{tx}$ is invertible and $t \ne 0$.
For the case where $t = 0$, let $(y, u) \in T_{(x, 0)}(\g \times \R) = \g \times \R$ be such that $d\phi_{(x, 0)}(y, u) = 0$.
Since $d\pi_\G(d\phi_{(x, 0)}(y, u)) = u$, we have $u = 0$.
It follows that $d\phi_{(x, 0)}(y, u) = y \in T_{(x, 0)}(\g \times \{0\}) = \g$, so $y = 0$.
\end{proof}

Note that $G$ acts naturally on $\G$ by conjugation on the first factor of $G \times \R^\times$ and the adjoint action on the first factor of $\g \times \{0\}$.
We call this the \defn{conjugation action} of $G$ on $\G$.

\begin{Prop}\label{xppf45jj}
The conjugation action of $G$ on $\G$ is smooth.
\end{Prop}

\begin{proof}
The action is clearly smooth on the open subset $G \times \R^\times$.
Since the map $\phi$ of Lemma \ref{15sdkw85} is equivariant with respect to the adjoint action on the first factor of $\g \times \R$, the conjugation action on $\G$ is also smooth on a neighborhood of the zero fiber.
\end{proof}

The following result will be useful for our discussion on the fusion product.

\begin{Prop}\label{haurfcai}
The map
\[
m : \G \times_\R \G \too \G,
\quad
((a, t), (b, t)) \mtoo \begin{cases}
(ab, t) & \text{if } t \ne 0 \\
(a + b, 0) & \text{if } t = 0
\end{cases}
\]
is smooth.
\end{Prop}

\begin{proof}
Since $m$ is clearly smooth away from the zero fiber, it is enough to show that it is smooth on a neighborhood of the set $\{((x, 0), (y, 0)) : x, y \in \g\}$.
Let $W \subset (\g \times \R) \times_\R (\g \times \R)$ be the set of $((x, t), (y, t))$ such that the Baker--Campbell--Hausdorff series $\mathrm{BCH}(tx, ty)$ converges.
Then the map
\[
\tilde{m} : W \too \g \times \R, \quad ((x, t), (y, t)) \mtoo
\begin{cases}
\frac{1}{t}\mathrm{BCH}(tx, ty) & \text{if } t \ne 0 \\
x + y & \text{if } t = 0
\end{cases}
\]
is smooth.
The statement then follows from Lemma \ref{15sdkw85} and the fact that the diagram 
\[
\begin{tikzcd}
W \arrow{r}{\tilde{m}} \arrow[swap]{d}{\phi \times \phi} & \g \times \R \arrow{d}{\phi} \\
\G \times_\R \G \arrow{r}{m} & \G
\end{tikzcd}
\]
commutes.
\end{proof}

Let $\theta^L$ and $\theta^R$ be the left and right invariant Maurer--Cartan forms on $G$.
We use left trivialisation to identify $T_gG$ with $\g$ for all $g \in G$.
Thus, $\theta^L_g(x) = x$ and $\theta^R_g(x) = \Ad_gx$ for all $(g, x) \in G \times \g = TG$.

\begin{Prop}\label{1487h8tl}
There are unique smooth sections $\widehat{\theta}^L$ and $\widehat{\theta}^R$ of $(\wedge^2 (\ker d\pi_\G)^*) \otimes \g \to \G$ such that $\widehat{\theta}^L|_{\G_t} = \frac{1}{t}\theta^L$ and $\widehat{\theta}^R|_{\G_t} = \frac{1}{t}\theta^R$ for all $t \ne 0$, and $\widehat{\theta}^L|_{\G_0} = \widehat{\theta}^R|_{\G_0}$ is the identity map $T_x\g = \g \to \g$ for all $x \in \g$.
\end{Prop}

\begin{proof}
We prove the existence of $\widehat{\theta}^L$; the existence of $\widehat{\theta}^R$ is done similarly.
Define $\widehat{\theta}^L$ by $\frac{1}{t}\theta^L$ on $G \times \R^\times \subset \G$; we need to show that $\widehat{\theta}^L$ extends smoothly to $\mathrm{Id}_\g$ at $t = 0$.
Let $\phi$ and $U$ be as in Lemma \ref{15sdkw85} and let $V \coloneqq \phi(U) \subset \G$.
Note that $\pi_\G \circ \phi = \pr_\R : \g \times \R \to \R$ is the projection onto the second factor.
It is then enough to show that $\phi^*\widehat{\theta}^L$ extends smoothly to $\mathrm{Id}_\g$ at $t = 0$ as a section of $(\wedge^2 \ker d\pr_\R) \otimes \g$.
For all $(x, t) \in \g \times \R^\times$ and $v \in \ker((d\pr_\R)_{(x, t)}) = \g$, we have
\begin{align*}
(\phi^*\widehat{\theta}^L)_{(x, t)}(v) &= \frac{1}{t} \theta^L(d\exp_{tx}(tv)) = d\exp_{tx}(v),
\end{align*}
which indeed extends smoothly to $\mathrm{Id}_\g$ as $t \to 0$.
\end{proof}

\subsection{Hamiltonian and quasi-Hamiltonian spaces}

A \defn{quasi-Hamiltonian space} will be denoted as a tuple $(M, \omega, \mu, G)$, where $M$ is a smooth manifold, $\omega$ is a 2-form, $\mu : M \to G$ is a smooth map, and $G$ is a Lie group equipped with an invariant non-degenerate symmetric bilinear form $\langle \cdot, \cdot \rangle$ on its Lie algebra $\g$, and satisfying the properties listed in \cite[Section 2.2]{alekseev1998lie}.
Note that for all $t \in \R^\times$, $(M, \frac{1}{t} \omega, \mu, G)$ is a quasi-Hamiltonian space with respect to $\frac{1}{t}\langle \cdot, \cdot \rangle$.

Similarly, a \defn{Hamiltonian space} is a tuple $(N, \sigma, \nu, G)$, where $\sigma$ is a symplectic form on $N$ and $\nu : N \to \g^* \cong \g$ is a moment map for the action of $G$ on $N$.

\subsection{Deformations of quasi-Hamiltonian spaces}

Equipped with these notions, we are now ready to make the main definition of this paper.

\begin{Def}\label{swcoikg4}
A \defn{Hamiltonian deformation} of a quasi-Hamiltonian space $(M, \omega, \mu, G)$ to a Hamiltonian space $(N, \sigma, \nu, G)$ is a smooth manifold $\D$ together with 
\begin{enumerate}[label={(\roman*)}]
\item\label{d5uauarr}
a smooth action of $G$ on $\D$,
\item
a $G$-invariant smooth submersion $\pi : \D \to \R$ with $[0, 1] \subset \operatorname{im} \pi$,
\item
a $G$-equivariant smooth map $\widehat{\mu} : \D \to \G$ (with respect to the conjugation action of Proposition \ref{xppf45jj}) such that $\pi_{\mathcal{G}} \circ \widehat{\mu} = \pi$, and
\item\label{v0b6urvr}
a smooth section $\widehat{\omega}$ of the vector bundle $\wedge^2(\ker d\pi)^* \to \D$, 
\end{enumerate}
such that the following conditions hold with respect to the fiber $\D_t \coloneqq \pi^{-1}(t)$, the 2-form $\omega_t \coloneqq \widehat{\omega}|_{\D_t}$ on $\D_t$, the map $\mu_t \coloneqq \widehat{\mu}|_{\D_t}$, and the $G$-action on $\D_t$ coming from \ref{d5uauarr}:
\begin{itemize}
\item $(\D_t, \widehat{\omega}_t, \widehat{\mu}_t, G)$ is a quasi-Hamiltonian space for all $t \ne 0$ and a Hamiltonian space for $t = 0$;
\item
$\D_1 \cong M$ as quasi-Hamiltonian spaces;
\item
$\D_0 \cong N$ as Hamiltonian spaces.
\end{itemize}
\end{Def}

\begin{Rem}
We could alternatively define a Hamiltonian deformation by replacing the section $\widehat{\omega}$ of $\wedge^2(\ker d\pi^*)$ in \ref{v0b6urvr} with a 2-form on $\mathcal{D}$.
In the smooth category, the two approaches are equivalent.
Indeed, given a section $\widehat{\omega}$, we can define a 2-form $\widehat{\omega}$ by $\widehat{\omega} \circ \mathrm{pr}_{\ker d\pi}$, where $\mathrm{pr}_{\ker d\pi}$ is the orthogonal projection to $\ker d\pi$ with respect to any Riemannian metric on $\mathcal{D}$.
\end{Rem}

\begin{Rem}
See \cite{Devalapurkar} for a related notion in algebraic geometry.
\end{Rem}

\section{Deformations of fusions and reductions}

The goal of this section is to show that Hamiltonian deformations are compatible with fusion products and reductions.

\subsection{Deformations of fusion products}

Recall that the \defn{internal fusion} \cite[Section 3]{alekseev1998lie} of a quasi-Hamiltonian space $(M, \omega, \mu, G \times G \times H)$ is the quasi-Hamiltonian space $(M, \tilde{\omega}, \tilde{\mu}, G \times H)$, where $\tilde{\omega} = \omega + \frac{1}{2}\langle\mu_1^*\theta^L \wedge \mu_2^*\theta^R\rangle$, $\tilde{\mu} = (\mu_1\mu_2, \mu_3)$, and $G$ acts diagonally.

Similarly, for a Hamiltonian space $(N, \sigma, \nu, G \times G \times H)$, we define the \defn{internal fusion} as the Hamiltonian space $(N, \sigma, \tilde{\nu}, G \times H)$, where $\tilde{\nu} = (\nu_1 + \nu_2, \nu_3)$ and $G$ acts diagonally.

One of the tools we will use to construct new examples of deformations is the following.

\begin{Thm}\label{jtwx49l3}
Let $(\D, \widehat{\omega}, \widehat{\mu}, G \times G \times H)$ be a Hamiltonian deformation of $(M, \omega, \mu, G \times G \times H)$ to $(N, \eta, \nu, G \times G \times H)$.
Then there is a natural Hamiltonian deformation of the internal fusion $(M, \tilde{\omega}, \tilde{\mu}, G \times H)$ to the internal fusion $(N, \sigma, \tilde{\nu}, G \times H)$.
\end{Thm}

\begin{proof}
We keep the same manifold $\D$ with the submersion $\pi : \D \to \R$.
By Proposition \ref{1487h8tl} and Proposition \ref{haurfcai}, the section $\widehat{\omega} + \frac{t}{2}\langle \widehat{\mu}_1^* \widehat{\theta}^L \wedge \widehat{\mu}_2^* \widehat{\theta}^R \rangle$ and the map $(m \circ (\widehat{\mu}_1 \times \widehat{\mu}_2), \widehat{\mu}_3) : \D \to \G \times \H$ have the desired properties.
\end{proof}

Recall also that given two quasi-Hamiltonian spaces $(M_i, \omega_i, G \times H_i)$ for $i = 1, 2$, their \emph{fusion} is the quasi-Hamiltonian space $M_1 \circledast M_2$ obtained by fusing the two $G$-factors of the quasi-Hamiltonian $(G \times H_1 \times G \times H_2)$-space $M_1 \times M_2$.
Similarly, we define the fusion $N_1 \circledast N_2$ of two Hamiltonian spaces $(N_i, \sigma_i, G \times H_i)$ for $i = 1, 2$ as $N_1 \times N_2$ with the action of $G \times H_1 \times H_2$ given by the diagonal copy of $G$.

\begin{Cor}\label{6mj8ptsk}
If the quasi-Hamiltonian spaces $(M_i, \omega_i, \mu_i, G \times H_i)$ admit Hamiltonian deformations to Hamiltonian spaces $(N_i, \sigma_i, \nu_i, G \times H_i)$, then $M_1 \circledast M_2$ has a natural Hamiltonian deformation to $N_1 \circledast N_2$.
\end{Cor}

\subsection{Deformations of reductions}

For the purpose of this section, suppose that the group $G$ is compact.
Let $(M, \omega, \mu, G \times H)$ be a quasi-Hamiltonian space.
Recall \cite[Section 5]{alekseev1998lie} that for any $f \in G$ such that $G_f \coloneqq \{g \in G : gfg^{-1} = f\}$ acts freely on $\mu_1^{-1}(f)$, the quotient
\[
M_f \coloneqq \mu_1^{-1}(f) / G_f
\]
is a quasi-Hamiltonian $H$-space whose 2-form descends from the pullback of $\omega$.

Similarly, if $(N, \sigma, \nu, G \times H)$ is a Hamiltonian space and $x \in \g$ is such that $G_x \coloneqq \{g \in G : \Ad_gx = x\}$ acts freely on $\nu^{-1}(x)$, then the Marsden--Weinstein--Meyer \cite{marsden1974reduction,MR331427} symplectic reduction $N_x \coloneqq \nu_1^{-1}(x)/G_x$ is a Hamiltonian $H$-space.

\begin{Thm}\label{oi2khfoe}
Let $(\D, \widehat{\omega}, \widehat{\mu}, G \times H)$ be a Hamiltonian deformation of $(M, \omega, \mu, G \times H)$ to $(N, \eta, \nu, G \times H)$.
Let $x \in \g$ and let $I \subset \R$ be a neighborhood of $[0, 1]$.
Suppose that 
\begin{enumerate}[label={\textup{(\roman*)}}]
\item\label{jkmndrxc}
$d\exp_{tx}$ is invertible for all $t \in I$,
\item\label{pnn277gb}
$G_x$ acts freely on $\nu_1^{-1}(x)$ and $\mu_1^{-1}(e^{tx})$ for all $t \in I \setminus \{0\}$, and
\item\label{i4svqrm1}
the curve $\phi_x : I \to \G$, $t \mapsto \phi(x, t)$ lifts to a curve $\widehat{\phi}_x : I \to \D$ such that $\widehat{\mu}_1 \circ \widehat{\phi}_x = \phi_x$.
\end{enumerate}
Set $\mathcal{X} \coloneqq \phi(\{x\} \times I) = \{(e^{tx}, t) : t \in I \setminus \{0\}\} \sqcup \{(x, 0)\}$.
Then $\widehat{\mu}^{-1}(\mathcal{X}) / G_x$ has the structure of a Hamiltonian deformation from $M_{e^x}$ to $N_x$.
\end{Thm}

\begin{proof}
Condition \ref{jkmndrxc} implies that $\mathcal{X}$ is a smooth submanifold of $\G$.
We show that $\mathcal{X}$ is transverse to $\widehat{\mu}_1$.
Let $t \in I \setminus \{0\}$ and $p \in \D$ be such that $\widehat{\mu}_1(p) = (e^{tx}, t)$.
Let $(y, u) \in T_{(e^{tx}, t)}\G = \g \times \R$.
Since $e^{tx}$ is a regular value of $\mu_1$, there exists $v \in \ker d\pi_p$ such that $d\widehat{\mu}_1(v) = (y - ux, 0)$.
It follows that $(y, u) = d\widehat{\mu}_1(v) + d\phi_{(x, t)}(u)$, proving transversality away from the zero fiber.
Similarly, consider $p \in \D$ such that $\widehat{\mu}_1(p) = (x, 0)$ and let $(y, u) \in T_{(x, 0)}(\g \times \R) = \g \times \R$.
We must find $v \in T_p\D$ and $w \in \R$ such that $d\phi_{(x, 0)}(y, u) = d\widehat{\mu}_1(v) + d\phi_{(x, 0)}(0, w)$.
Setting $w = 0$, this is equivalent to $d\widehat{\mu}_1(v) = d\phi_{(x, 0)}(y, 0) = y \in T_{(x, 0)}(\g \times \{0\}) = \g$.
We can then use that $\nu_1$ is regular at $x$ to find $v \in \ker d\pi_p$ such that $d\widehat{\mu}_1(v) = y$.
Hence, $\iota : \widehat{\mu}^{-1}(\mathcal{X}) \hookrightarrow \D$ is a smooth submanifold and $\overline{\D} \coloneqq \widehat{\mu}^{-1}(\mathcal{X}) / G_x$ has the structure of a smooth manifold such that the quotient map is a smooth submersion.

Since $\pi : \D \to \R$ is $G$-invariant, it descends to a smooth map $\overline{\pi} : \overline{\D} \to \R$.
To show that $\overline{\pi}$ is a submersion, let $u \in T_{\overline{\pi}(p)}\R$ for $p \in \widehat{\mu}_1^{-1}(\mathcal{X})$.
Let $v = d\widehat{\phi}_x(u)$.
Then $d\widehat{\mu}_1(v) = d\phi_{\overline{\pi}(p)}(u) \in T\mathcal{X}$, so $v$ is a tangent vector of $\widehat{\mu}_1^{-1}(\mathcal{X})$.
Let $\pi_\G : \G \to \R$ be the submersion of $\G$.
Then $d\pi(v) = d\pi_\G(d\widehat{\mu}_1(v)) = d\pi_\G(d\phi_x(u)) = u$.
It follows that $\pi|_{\widehat{\mu}_1^{-1}(\mathcal{X})}$ is a submersion, and hence so is $\overline{\pi}$.

Take any $G$-invariant Riemannian metric on $\D$ and let $\pr_{\ker d\pi} : T\D \to \ker d\pi$ be the corresponding orthogonal projection.
Then the 2-form $\eta \coloneqq \widehat{\omega} \circ \pr_{\ker d\pi}$ is $G$-invariant and basic with respect to the action of $G_x$ (i.e.\ $i_{v}\widehat{\omega}' = 0$ for every vector field $v$ on $\D$ generated by the $G_x$-action).
It follows that $\eta$ descends to a $2$-form $\bar{\eta}$ on $\overline{\D}$ whose restriction to each fiber is the reduced form.
We may then define $\widehat{\bar{\omega}}$ as the restriction of $\bar{\eta}$ to $\ker d\bar{\pi}$.

The moment map $\widehat{\mu}_2$ is $G$-invariant and hence descends to $\overline{\D}$.
Similarly, the $H$-action on $\mathcal{D}$ descends to $\overline{\D}$ since it commutes with the $G$-action.
\end{proof}

\section{Examples of Hamiltonian deformations}

Throughout this section, $G$ is a Lie group equipped with an invariant non-degenerate symmetric bilinear form $\langle \cdot, \cdot \rangle$ on its Lie algebra $\g$.
We also let $\G$ be the deformation of $G$ to $\g$ as in \eqref{reboyrqr}.

\subsection{The double of a Lie group}
Recall that the \defn{double} \cite[Section 3.2]{alekseev1998lie} of $G$ is the quasi-Hamiltonian space $(D(G), \omega, \mu, G \times G)$, where $D(G) = G \times G$ with the action of $G \times G$ by
\begin{equation}\label{so5i20z6}
(g,h)\cdot(\alpha,\beta) = (g\alpha h^{-1},h\beta h^{-1}),
\end{equation}
the 2-form
\[
\omega \coloneqq \frac{1}{2}\langle \text{Ad}_{\beta}\alpha^{*}\theta^{L} \wedge \alpha^{*}\theta^{L} \rangle + \frac{1}{2} \langle \alpha^{*}\theta^{L} \wedge (\beta^{*}\theta^{L} + \beta^{*}\theta^{R}) \rangle,
\]
where $\alpha, \beta : G \times G \to G$ also denote the two projections, and the moment map
\[
\mu : G\times G \longrightarrow G \times G,
\quad
(\alpha,\beta) \longmapsto (\alpha\beta\alpha^{-1},\beta^{-1}).
\]
Recall also that the cotangent bundle $T^*G$ has a natural structure of a Hamiltonian $G \times G$-space with respect to the canonical symplectic form and the action by left and right translations.
By using the trivialization by left translations $T^*G = G \times \g^*$ and identifying $\g^*$ with $\g$ via the bilinear form, this Hamiltonian space takes the form $(G \times \g, \sigma, \nu, G \times G)$, where the action is
\begin{equation}\label{wx1ma5ha}
(g, h) \cdot (\alpha, x) = (g\alpha h^{-1}, \Ad_hx)
\end{equation}
the symplectic form is
\begin{equation}\label{b8cj1rl6}
\sigma_{(\alpha, x)}((y_1, z_1), (y_2, z_2)) = \langle y_1, z_2 \rangle - \langle y_2, z_1 \rangle + \langle x, [y_1, y_2] \rangle,
\end{equation}
for $(y_i, z_i) \in \g \times \g \cong T_{(\alpha, x)}(G \times \g)$, and the moment map is
\[
\nu(\alpha, x) = (\Ad_\alpha x, -x).
\]

\begin{Thm}\label{ikr142qd}
The set $\D \coloneqq G \times \G$ has the structure of a Hamiltonian deformation from the double $D(G)$ to the cotangent bundle $T^*G$.
\end{Thm}

\begin{proof}
Let $\pi : \D \to \R$ be the composition of the projections $G \times \G \to \G$ and $\G \to \R$.
Note that $\D_t = D(G)$ for $t \ne 0$ and $\D_0 = G \times \g = T^*G$.
By Proposition \ref{xppf45jj}, the actions \eqref{so5i20z6} and \eqref{wx1ma5ha} induce a smooth action of $G \times G$ on $\D$.
Let $\widehat{\mu} : \D \to \G \times \G$ be given by $\mu$ on $G \times G \times \{t\}$ for $t \ne 0$ and by $\nu$ on $G \times \g \times \{0\}$.
Smoothness of $\widehat{\mu}$ follows from Lemma \ref{15sdkw85} and the commutativity of the following diagram
\[
\begin{tikzcd}
G \times \g \times \R \arrow{r} \arrow[swap]{d}{\mathrm{Id}_G \times \phi} & \g \times \R \times \g \times \R \arrow{d}{\phi \times \phi} & (\alpha, x, t) \arrow[mapsto]{r} & (\Ad_\alpha x, t, -x, t)\\
\D \arrow{r}{\widehat{\mu}} & \G \times \G.
\end{tikzcd}
\]
It remains to show that the section $\widehat{\omega} \coloneqq \frac{1}{t}\omega$ of $\wedge^2 (\ker d\pi)^*$ over the open set $G \times G \times \R^\times \subset \D$ extends smoothly to $\D$ in such a way that $\widehat{\omega}|_{\D_0} = \sigma$, where $\sigma$ is the symplectic form \eqref{b8cj1rl6}.
Let $\phi$ and $U$ be as in Lemma \ref{15sdkw85} and let $V \coloneqq \phi(U) \subset \G$.
The map
\[
\psi = \mathrm{Id}_G \times \phi : G \times \g \times \R \too \D
\]
then restricts to a diffeomorphism $\psi : G \times U \to G \times V$ such that $\pi \circ \psi = \pr_\R$.
Note also that $\psi$ restricts to the identity map on $G \times \g \times \{0\}$.
It then suffices to show that the section $\psi^*\widehat{\omega}$ of $\wedge^2 \ker d\pr_\R$ extends smoothly to $\sigma$ at $t = 0$.
Let $(\alpha, x, t) \in G \times \g \times \R^\times$ and let $(y_i, z_i) \in \ker (d\pr_\R)_{(\alpha, x, t)} = \g \times \g$ for $i = 1, 2$.
Then
\begin{align*}
&(\psi^*\widehat{\omega})_{(\alpha, x, t)}((y_1, z_1), (y_2, z_2)) \\
&=
\frac{1}{t}\omega_{(\alpha, e^{tx})}((y_1, d\exp_{tx}(tz_1)), (y_2, d\exp_{tx}(tz_2))) \\
&=
\frac{1}{2t}\left(\langle \Ad_{e^{tx}}y_1, y_2 \rangle - \langle \Ad_{e^{tx}} y_2, y_1 \rangle\right) \\
&\quad + \frac{1}{2t}\left(\langle y_1, d\exp_{tx}(tz_2) + \Ad_{e^{tx}}d\exp_{tx}(tz_2) \rangle - \langle y_2, d\exp_{tx}(tz_1) + \Ad_{e^{tx}}d\exp_{tx}(tz_1) \rangle \right) \\
&=
\langle y_1, \tfrac{\Ad_{e^{-tx}} - \Ad_{e^{tx}}}{2t}y_2 \rangle  + \frac{1}{2}(\langle y_1, d\exp_{tx}(z_2) + \Ad_{e^{tx}} d\exp_{tx}(z_2) \rangle - \langle y_2, d\exp_{tx}(z_1) + \Ad_{e^{tx}} d\exp_{tx}(z_1)\rangle),
\end{align*}
which indeed extends smoothly to $\sigma$ as $t \to 0$.
\end{proof}

\subsection{Conjugacy classes}

Recall that every conjugacy class $C_a \coloneqq \{gag^{-1} : g \in G\}$ for $a \in G$ is a quasi-Hamiltonian space with respect to the conjugation action \cite[Section 3.1]{alekseev1998lie}.
The moment map is the inclusion $\mu : C_a \hookrightarrow G$ and the 2-form is given by
\[
\omega_f(X_{v_1}, X_{v_2}) = \dfrac{1}{2}\left(\langle v_1, \Ad_f v_2 \rangle - \langle v_2, \Ad_f v_1 \rangle \right)
\]
for all $f \in C_a$ and vector fields $X_{v_1}, X_{v_2}$ generated by $v_1, v_2 \in \g$.

Similarly, every adjoint orbit $\O_x \coloneqq \{\Ad_g x : g \in G\}$ for $x \in \g$ is a Hamiltonian space with respect to the adjoint action, the inclusion $\nu : \O \hookrightarrow \g \cong \g^*$, and the 2-form
\begin{equation}\label{mz13xpwj}
\sigma_y(X_{v_1}, X_{v_2}) = -\langle y, [v_1, v_2] \rangle
\end{equation}
for all $y \in \O_x$ and $X_{v_1}, X_{v_2}$ the vector fields generated by $v_1, v_2 \in \g$.

Let $x \in \g$ be such that $d\exp_{tx}$ is invertible for all $t \in [0, 1]$.
We will construct a Hamiltonian deformation from $C_{e^x}$ to $\O_x$.
Let $I \subset \R$ be the open interval such that $d\exp_{tx}$ is invertible for all $t \in I$.
Let $\C$ be the subset of $\G = (G \times \R^\times) \sqcup (\g \times \{0\})$ given by the union of $C_{e^{tx}} \times \{t\}$ for $t \in I \setminus \{0\}$ and $\O_x \times \{0\}$.

\begin{Thm}\label{thnlzgh9}
The set $\C$ has the structure of a Hamiltonian deformation from $C_{e^x}$ to $\O_x$.
\end{Thm}

\begin{proof}
The map $\phi : \g \times \R \to \G$ of Lemma \ref{15sdkw85} restricts to a diffeomorphism from $\g \times I$ to an open subset of $\G$ containing $\C$.
Since $\phi^{-1}(\C) = \O_x \times I$, we see that $\C$ is a smooth submanifold of $\G$.
Moreover, this shows that the composition $\pi_\C : \C \hookrightarrow \G \to \R$ is a smooth submersion whose image is $I$, and hence contains $[0, 1]$.
The diffeomorphism $\phi : \O_x \times I \to \C$ is equivariant with respect to conjugation, so this also shows that the conjugation action on $\C$ is smooth.
We let $\widehat{\mu} : \C \hookrightarrow \G$ be the inclusion map.
It remains to show that the section $\widehat{\omega} \coloneqq \frac{1}{t}\omega_{e^{tx}}$ of $\wedge^2 (\ker d\pi_\C)^*$ over $\pi_\C^{-1}(\R^\times)$ extends smoothly to $\C$ in such a way that $\widehat{\omega}|_{\C_0}$ is the symplectic form $\sigma$ in \eqref{mz13xpwj}.
Note that $\phi$ restricts to the identity map on $\O_x \times \{0\}$ so, as in the proof of Theorem \ref{ikr142qd}, it suffices to show that $\phi^*\widehat{\omega}$ extends smoothly to $\sigma$ at $t = 0$.
First note that for all $(y, t) \in \O_x \times \R^\times$ and $(X_v|_y, 0) \in \ker (d\pr_\R)_{(y, t)} \subset T_{(y, t)}(\O_x \times I)$, where $v \in \g$, we have
\[
d\phi_{(y, t)}(X_v|_y, 0) = (X_v|_{e^{ty}}, 0).
\]
Hence, for all $(X_{v_i}|_y, 0) \in \ker (d\pr_\R)_{(y, t)}$ with $i = 1, 2$ and $t \ne 0$, we have
\begin{align*}
(\phi^*\widehat{\omega})_{(y, t)}((X_{v_1}|_y, 0), (X_{v_2}|_y, 0))
&=
\frac{1}{t} \omega_{e^{ty}}(X_{v_1}, X_{v_2}) \\
&=
\frac{1}{2t}
\left(\langle v_1, \Ad_{e^{ty}} v_2 \rangle - \langle v_2, \Ad_{e^{ty}} v_1 \rangle\right) \\
&=
\langle v_1, \frac{\Ad_{e^{ty}} - \Ad_{e^{-ty}}}{2t} v_2 \rangle,
\end{align*}
which indeed extends smoothly to \eqref{mz13xpwj} as $t \to 0$.
\end{proof}

\subsection{Moduli spaces of flat connections}

For a compact connected oriented surface $\Sigma$ of genus $g$ with $r + 1$ boundary components ($r \ge 0$), let $\M(\Sigma)$ be the moduli space of flat connections on the trivial principal $G$-bundle over $\Sigma$ \cite{MR702806}.
Recall \cite[Theorem 9.1]{alekseev1998lie} that $\M(\Sigma)$ has the structure of a quasi-Hamiltonian $G^{r+1}$-space.
More precisely, there is an identification
\[
\M(\Sigma) = \underbrace{D(G) \circledast \cdots \circledast D(G)}_{\text{$r$ times}} \circledast \underbrace{\mathbf{D}(G) \circledast  \cdot \cdots \circledast \mathbf{D}(G)}_{\text{$g$ times}},
\]
where $\mathbf{D}(G)$ is the internal fusion of $D(G)$.
We can define an additive version of this space as
\[
\N(\Sigma) = \underbrace{T^*G \circledast \cdots \circledast T^*G}_{\text{$r$ times}} \circledast \underbrace{\mathbf{T^*}G \circledast  \cdot \cdots \circledast \mathbf{T^*}G}_{\text{$g$ times}},
\]
where $\mathbf{T}^*G$ is $T^*G$ with the diagonal action of $G$.
By Theorem \ref{jtwx49l3}, Corollary \ref{6mj8ptsk}, and Theorem \ref{ikr142qd}, we get:

\begin{Cor}\label{lstnw1kj}
There is a Hamiltonian deformation from $\M(\Sigma)$ to $\N(\Sigma)$.
\end{Cor}

More generally, let $x_0, \ldots, x_r \in \g$ be such that the conditions of Theorem \ref{oi2khfoe} are satisfied and consider the conjugacy classes $C_i \coloneqq C_{e^{x_i}}$.
Recall \cite[Theorem 9.2]{alekseev1998lie} that the moduli space $\M(\Sigma, C_0, \ldots, C_r)$ of flat connections on the trivial principal $G$-bundle over $\Sigma$ whose holonomy on the $i$th boundary has values in $C_i$ is a quasi-Hamiltonian reduction of $\M(\Sigma)$ with respect to $(C_0, \ldots, C_r)$.
By Theorem \ref{oi2khfoe}, we have the following.

\begin{Cor}\label{rupcqbmj}
There exists a Hamiltonian deformation from $\M(\Sigma,C_0, \ldots, C_r)$ to a symplectic reduction of $\N(\Sigma)$.
\end{Cor}
 
\bibliographystyle{plain}
\bibliography{deformations-of-quasi-hamiltonian-spaces.bib}

\end{document}